\documentclass[a4paper,11pt, twoside, leqn]{article}

\usepackage{amsmath}
\usepackage{amsfonts}
\usepackage{amssymb}
\usepackage{dsfont}
\usepackage{amsthm}
\usepackage{stackrel}
\usepackage[nobysame]{amsrefs}
\usepackage[cp1250]{inputenc}
\usepackage{setspace}
\usepackage[T1]{fontenc}
\usepackage{geometry}
\usepackage{fancyhdr}
\usepackage{graphicx}
\usepackage{color}
\usepackage{enumerate}
\usepackage{fullpage}

\newcommand{\eN}{{\ensuremath{\mathbb N}}}
\newcommand{\Ex}{{\ensuremath{\mathbb E}}}
\newcommand{\Pro}{{\ensuremath{\mathbb P}}}
\newcommand{\eR}{{\ensuremath{\mathbb R}}}
\newcommand{\lv}{\left\lVert}
\newcommand{\rv}{\right\rVert}
\newcommand{\1}{\mathds{1}}
\newcommand{\eps}{\varepsilon}

\newcommand{\wym}{r}

\newcommand{\parr}{\log_2 \kappa}
\newcommand{\odpar}{\frac{1}{\parr}}
\newcommand{\Slev}{SRV($\kappa$) }

\newcommand{\XX}{\mathcal{X}}
\newcommand{\XXb}{\bar{\mathcal{X}}}
\newcommand{\Ik}{\mathcal{I}(k)}
\newcommand{\Il}{\mathcal{I}(l)}
\newcommand{\Ikc}{\mathcal{I}_4(k)}
\newcommand{\Ilc}{\mathcal{I}_4(l)}

\newcommand{\Xk}{X^{(k)}}
\newcommand{\Xd}{X^{(\wym+1)}}
\newcommand{\Xld}{X^{(\leq \wym)}}
\newcommand{\mia}{\max_i \lv (a_{ij})_j \rv_2}
\newcommand{\mja}{\max_j \lv (a_{ij})_i \rv_2}
\newcommand{\RX}{R_X(A)}
\newcommand{\AT}{ (a_{ij}t_{ij})_{i,j\leq n} }
\newcommand{\RTX}{ \sup_{t\in T} R_X\big((a_{ij}t_{ij})_{i,j \leq n}\big) }
\newcommand{\MTI}{\max_i \sup_{t\in T} \lv (a_{ij} t_{ij})_{j} \rv_2 }
\newcommand{\MTJ}{\max_j \sup_{t\in T} \lv (a_{ij}t_{ij})_i \rv_2}
\newcommand{\Log}{\mathit{Log}}
\newcommand{\rown}{\approx}
\newcommand{\AX}{(a_{ij}X_{ij})_{i,j\leq n}}

\newtheorem{theorem}[subsection]{Theorem}
\newtheorem{lemma}[subsection]{Lemma}

\newtheorem{fact}[subsection]{Fact}
\newtheorem{corollary}[subsection]{Corollary}
\newtheorem{remark}[subsection]{Remark}

\providecommand{\keywords}[1]{\textbf{\textit{Keywords: }} #1}

\providecommand{\klas}[1]{\textbf{\textit{AMS MSC 2010: }} #1}

\begin{document}
\author{Rafa{\l} Meller\footnote{Institute of Mathematics, University of Warsaw, Banacha 2, 02-097 Warsaw, Poland, 
    rmeller@mimuw.edu.pl}\\
Institute of Mathematics\\
University of Warsaw\\
02-097 Warszawa, Poland\\
E-mail: r.meller@mimuw.edu.pl}
\title{Spectral norm of matrices with independent entries up to polyloglog.\thanks{Research supported by the by National Science Centre, Poland grant 2021/40/C/ST1/00330.}}
\date{}
\maketitle

\renewcommand{\thefootnote}{}



\renewcommand{\thefootnote}{\arabic{footnote}}
\setcounter{footnote}{0}


\begin{abstract}
In this paper, the expectation of the operator norm of the random matrix $(a_{ij} X_{ij})_{i,j\leq n}$ is studied, under assumption that $(X_{ij})_{i,j \leq n}$ are independent random variables that satisfy $\lv X_{ij} \rv_{2p} \leq \kappa \lv X_{ij} \rv_p$ for each $p\geq 1$. An upper bound is derived in terms of quantities that admit a relatively simple analytic formula.
 Our upper bound yields two-sided bound up to a factor given by a power of an iterated logarithm. This factor is considerably smaller than the natural scale of the problem. This provides positive evidence for a conjecture formulated by Lata{\l}a and \'{S}wi\k{a}tkowski.

 \\
\klas{60B20, 46B09}
\end{abstract}

\section{Introduction}

Estimating the operator norm (denoted by $\| \cdot \|_{op}$) is one of the classical
problems in random matrix theory. If the random matrix is ``highly'' symmetric
(for example, the entries are i.i.d.), then this question is well understood,
cf.~\cite{mac:los}. The issue becomes more complicated in the inhomogeneous
setting. For instance, we consider i.i.d. random variables
$(X_{ij})_{i,j \leq n}$, a deterministic matrix $A=(a_{ij})_{i,j \leq n}$, and we are interested in obtaining an upper bound for
\begin{equation*}
 \Ex \lv (a_{ij} X_{ij})_{i,j \leq n} \rv_{op}=\Ex \sup_{v,w \in B_2} \sum_{ij} a_{ij} X_{ij} v_i w_j.   
\end{equation*}
Here and subsequently $B_2$ is the Euclidean ball in $\eR^n$. The most favorable case is when we are able to obtain two-sided estimates i.e. we can show that
\[\Ex \lv (a_{ij} X_{ij})_{i,j \leq n} \rv_{op}\rown F((a_{i,j})_{i,j}, (\mathcal{L}(X_{ij}))_{i,j}),\]
where $F$ is ''simple''. For two nonnegative functions $f,g$, we write $f\lesssim g$ (or $f \lesssim^\alpha g$) if there exists an absolute constant $C$ (constant $C$, which depends only on the parameter $\alpha$) such that $f \leq Cg$ ($f \leq C(\alpha) g)$.  The notation $f\approx g$  means that $f \lesssim g$ and $g \lesssim g$. Analogously we define $f \rown^\alpha g$. We also use the convention that the constant $C$ (the constant $C(\alpha)$) may differ at each occurrence .\\
The Gaussian case has been well understood only recently. In \cite{lathan} Lata{\l}a, van Handel, and Youssef showed that, for any symmetric matrix $(a_{ij})_{i,j \leq n}$
\begin{equation}\label{eq:locwstepgaus}
  \Ex \lv (a_{ij} g_{ij})_{i,j \leq n} \rv_{op}\rown \max_i \lv (a_{ij})_j \rv_2 + \max_{ij} a'_{ij} \sqrt{\log i}.  
\end{equation}
Here $(a'_{ij})$ is obtained by permuting the rows and columns of the matrix $(a_{ij})$, such that, for all $j\leq n$,  $\max_j a'_{1,j} \geq \max_j a'_{2,j} \geq \cdots \geq \max_j a'_{n,j}$. The formula for non-symmetric matrices is analogous. In the same paper, they also derived an estimate for random variables satisfying the following moment condition
\begin{equation}\label{eq:wstepmom}
\exists_{\beta\geq \frac{1}{2}} \forall_{p\geq 1}\lv X_{ij} \rv_p:=\sqrt[p]{ \Ex |X_{ij}|^p } \rown p^\beta.    
\end{equation}
Condition $\beta \geq 1/2$ is crucial, since in this case, $(X_{ij})_{i,j \leq n}$'s are Gaussian mixtures (they are conditionally Gaussian random variables). By appropriate conditioning, the problem reduces to the Gaussian case, and formula \eqref{eq:locwstepgaus} yields two-sided estimates.\\
Seginer was one of the first to study the case of Rademacher random variables (i.e., symmetric random variables taking values $\pm1$). In \cite{seg} he showed that
\[\Ex \lv (a_{ij} \eps_{ij})_{i,j \leq n} \rv_{op}\lesssim \Log^{1/4} n \left(\max_i \lv (a_{ij})_j \rv_2 + \max_j \lv (a_{ij})_i \rv_2 \right),  \]
where $\Log(x):=\ln(x \vee e)$ and  $(\eps_{ij})_{i,j\in \eN}$ stands for a family of independent Rademacher random variables. We shall use this notation throughout the paper. He also proved that  $\Log^{1/4} n$ cannot be improved. Unfortunately, this estimate is often suboptimal (for instance, when the matrix is diagonal). Recently, some progress has been made, and it is now known that
\begin{equation}\label{eq:wsteprad}
    H \lesssim \Ex \lv (a_{ij} \eps_{ij})_{i,j \leq n} \rv_{op}\lesssim  \left( \Log \Log \Log \, n \right) H 
\end{equation}
where
\[H= \max_i \lv (a_{ij})_j \rv_2 + \max_j \lv (a_{ij})_i \rv_2 +\max_{1 \leq k \leq n} \min_{I\subset [n],|I|\leq k} \sup_{v,w\in B_2} \lv \sum_{i,j \notin I} a_{ij} \eps_{ij} v_i w_j \rv_{\Log(k)}.\]
The lower bound was proved by  Lata{\l}a and \'{S}wi\k{a}tkowski in~\cite{latswiat}, while the upper bound was proved by  Lata{\l}a in~\cite{latop}. In the latter, it was also shown that the triple logarithm is redundant when the matrix has entries in the set $\{0,1\}$. Although estimate \eqref{eq:wsteprad} is not two-sided, the triple logarithmic factor appearing in it is much smaller than the natural scaling for this problem, namely a power of logarithms. For this reason, it is very likely that estimate \eqref{eq:wsteprad} is two-sided, meaning the triple logarithmic factor is unnecessary.\\
To the best of our knowledge, the above results are the only ones that address two-sided (or almost two-sided) estimates for the operator norm of matrices
of the form $(a_{ij} X_{ij})_{i,j \le n}$. Rademacher random variables may be viewed as having very light tails, while random variables satisfying the moment condition \eqref{eq:wstepmom} including, in particular, Gaussian random variables, can be regarded as having medium or heavy tails.  We observe a gap in the existing results with respect to the distributions of the underlying random variables. To date, this regime has not been investigated, which is one of the aims of this paper. Subgaussian random variables are a prototypical example of distributions lying in this gap; for instance, random variables with tails
\[\Pro(|X_{ij}|\geq t) = e^{-|t|^\alpha}, \alpha >2.\]
We establish an upper bound for $\Ex \lv (a_{ij} X_{ij})_{i,j\leq n} \rv_{op}$ under the following moment condition
\begin{equation}\label{eq:*}
   \exists_{\kappa} \forall_{p\geq 1} \lv X \rv_{2p} \leq \kappa  \lv X \rv_{p}.
\end{equation}   
We  denote the class of all symmetric random variables that satisfy \eqref{eq:*} by \Slev. This class contains many natural random variables, such as normal, subgaussian, and Weibull, among others. The class also posses some desirable  properties, such as formulas for moments (cf. Theorem \ref{tw:formulamomenty}) and decomposition lemmas (cf. Lemma \ref{lem:rozbij}).\\
We now state the main result of this paper. Its essence is that the polyloglog factor is much smaller than the natural scale of the problem, namely a power of a logarithm
(as in Lata{\l}a's result, that is, in the upper bound in~\eqref{eq:wsteprad}). Therefore, Theorem \ref{tw:main} may be viewed as lending support to the hypothesis proposed by Lata{\l}a and \'{S}wi\k{a}tkowski that the lower bound in \eqref{eq:twglowneoszacowaniepodwojnylog} can be reversed (see \cite[Conjecture 4.3]{latswiat}).
\begin{theorem}\label{tw:main}
Assume that $(X_{ij})_{i,j,\leq n} $ are independent, symmetric random variables such that exists $\kappa>1$ such that for all $p\geq 1$ we have 
$\lv X_{ij} \rv_{2p} \leq \kappa \lv X_{ij} \rv_p$. Let also $\Ex |X_{ij}|=1$ for each $i,j \leq n$. Then for any matrix $A=(a_{ij})_{i,j\leq n}$ 
   \begin{equation}\label{eq:twglowneoszacowaniepodwojnylog}
         \Ex \lv (a_{ij} X_{ij})_{i,j \leq n} \rv_{op}\lesssim^\kappa   \Log^{C(\kappa)} \Log(n)\left(M(A)+\sup_{v,w \in B_2} \lv \sum_{ij} a_{ij} v_i w_j X_{ij}\rv_{\Log \, n}\right),    
        \end{equation}
       where $M(A):=\mia+\mja$.  In particular
        \begin{equation}\label{eq:wstepdwustronne}
          M(A)+D(A,X)\lesssim^\kappa \Ex \sup_{v,w \in B_2} \sum_{ij} a_{ij} X_{ij}v_i w_j \lesssim^\kappa \Log^{C(\kappa)} \Log(n)\left(M(A)+D(A,X)\right),
        \end{equation}
        where
        \[D(A,X):=\max_{1\leq k \leq n} \min_{I \subset [n],|I|\leq k} \sup_{v,w \in B_2} \lv \sum_{i,j \notin I} a_{ij} X_{ij}v_i w_j \rv_{\Log \, {k}}.\]
\end{theorem}


To shorten the notation we define
\begin{align}
 R_{(X_{ij})_{i,j\leq n}}(A,p)&=R_X(A,p)=\sup_{v,w \in B_2} \lv \sum_{ij} a_{ij} v_i w_j X_{ij}\rv_{p}.\label{eq:defR}
\end{align}
Most often, we will consider $p=\Log \, n$. In this case, we simply write $R_X(A)$. 

The parameter $R_X(A,p)$ can be estimated using the formula for moments
of random variables from the \Slev class (see Theorem~\ref{tw:formulamomenty} below). Before doing so, we need to introduce some additional notation.
 For fixed $(X_{ij})_{i,j \leq n}$ from the \Slev class, we define
 \begin{align*}
\hat{N}^X_{ij}(t):=\begin{cases}t^2 &|t|\leq 1 \\ -\ln \Pro(|X_{ij}|\geq t) &|t|>1, \end{cases}    
 \end{align*}
and
\begin{equation}\label{def:BPX}
    B^X_p=\{ s\in \eR^{n^2}: \sum_{i,j=1}^n \hat{N}^X_{ij}(s_{ij}) \leq p\}.
\end{equation}
For a clean statement, we temporarily change our normalization in the next theorem. The reason for this is given in the proof.
\begin{theorem}\label{twr:oszacujRX}
Under the assumptions of Theorem~\ref{tw:main}, except that we assume the normalization $\Ex |X_{ij}| = 1/e$ for $i,j \le n$, we have
\[\sup_{\substack{I\subset [n]\times [n]\\ |I|= p}} \sup_{\ t\in B^X_p} \lv (a_{ij}t_{ij})_{(i,j)\in I} \rv_{op}  \lesssim^\kappa R_X(A,p) \lesssim^\kappa \sqrt{\Log \, p} \sup_{\substack{I\subset [n]\times [n]\\ |I|= p}} \sup_{ t\in B^X_p} \lv (a_{ij}t_{ij})_{(i,j)\in I} \rv_{op}.  \]
In particular
\[R_X(A) \lesssim \sqrt{\Log \Log \, n}\sup_{\substack{I\subset [n]\times [n]\\ |I|=\Log\, n}} \left(  \sup_{t\in B^X_{\Log \, n}} \lv (a_{ij}t_{ij})_{(i,j)\in I} \rv_{op}   \right).\]
\end{theorem}

\begin{remark}
For any $|I|=p$ we have $\sum_{(i,j)\in I}\hat{N}^X_{ij}(1)=p$. Thus,
\begin{multline*}
\sup_{\substack{I\subset [n]\times [n]\\ |I|=p}} \sup_{\ t\in B^X_p} \lv (a_{ij}t_{ij})_{(i,j)\in I} \rv_{op} \\
\approx \sup_{\substack{I\subset [n]\times [n]\\ |I|=p}} \sup_{\ t\in B^X_p} \lv (a_{ij}t_{ij}\1_{|t_{ij}|>1})_{(i,j)\in I} \rv_{op}+ \sup_{\substack{I\subset [n]\times [n]\\ |I|=p}}\lv (a_{ij})_{(i,j)\in I}\rv_{op}.\end{multline*}
The second expression is easier to estimate, since we do not have to deal with the bracketing definition of the functions $(\hat{N}^X_{ij}(t))_{i,j \leq n}$.
\end{remark}

Theorem \ref{tw:main} easily implies bounds on moments of operator norm.
\begin{theorem}
Under the assumptions of Theorem~\ref{tw:main}, we have, for any $p \geq 1$,  
\[    \sqrt[p]{\Ex \lv (a_{ij} X_{ij})_{i,j \leq n} \rv^p_{op} } \lesssim^\kappa   \Log^{C(\kappa)} \Log(n)\big(M(A)+R_X(A,{\max(p,\Log\, n)})\big).  \]
\end{theorem}
\begin{proof}
 We recall \eqref{eq:defR}, the definition of $R_X(A,p)$. The assertion is an easy consequence of \cite[Theorem 1.1]{strzelec}.
\end{proof}

To prove Theorem \ref{tw:main}, we follow Lata{\l}a's argument which was used in \cite{latop}. We associate to a symmetric matrix $(a_{ij})_{i,j \leq n}$ a graph $G_A=([n],E_A)$, where $(i,j)\in E_A$ if and only if $i\neq j$ and $a_{i,j}\neq 0$.  We denote the maximal degree of vertices in $G_A$ by $d_A$. Lata{\l}a's idea was to obtain estimates on the operator norm that depend not on the matrix's dimension but on the (typically) much smaller parameter $d_A$.
\begin{theorem}\label{tw:glow}
 Assume that $(X_{ij})_{i,j \leq n}$ satisfies assumptions of Theorem \ref{tw:main}. Then for any symmetric matrix $A=(a_{ij})_{i,j \leq n}$  we have
\begin{equation}\label{eq:twglownezda}
\Ex \lv (a_{ij} X_{ij})_{i,j \leq n} \rv_{op} \lesssim^\kappa   \Log^{C(\kappa)}(d_A)\left(\mia +R_X(A) \right).     
\end{equation}
\end{theorem}

\textbf{Organization of the paper.} The next section discusses the basic tools used in the sequel. In Section \ref{sec:tw1.1} we deduce Theorem \ref{tw:main} from Theorem \ref{tw:glow}. The argument builds on a generalization of the proof of \cite[Theorem 1.9]{latop}. We also prove Theorem \ref{twr:oszacujRX}. The goal of Section~\ref{sec:subgaus} is to prove Theorem~\ref{prop:oszapodgaus}. It is a version of Theorem~\ref{tw:glow} for subgaussian random variables, with an additional supremum over a set $T \subset \eR^{n^2}$
(no conditions are imposed on~$T$). Our argument is a direct transcription of Lata{\l}a's proof of \cite[Proposition~4.4]{latop}, rewritten to fit our notation and assumptions.
The two proofs are otherwise identical. This is made possible by the highly general nature of Lata{\l}a's arguments, which rely on pointwise estimates for random variables, together with a clever discretizations of the Euclidean unit ball. However, not all of Lata{\l}a's arguments admit such a generalization. Consequently, we could not prove Theorem \ref{tw:main} with a triple logarithmic factor. We present the proof in full in order to keep the exposition self-contained. The last section is devoted to the proof of Theorem \ref{tw:glow}.\\
\textbf{Acknowledgments:} The author would like to thank Rafa{\l} Lata{\l}a for valuable discussions related to this work. This paper was inspired by reading his article \cite{latop}.

\section{Tools}
 We start with two inequalities, known as contraction principles in the literature. Here and subsequently $\eps_1,\eps_2,\ldots$ stands for independent Rademacher random variables (symmetric $\pm 1$ random variables). 
\begin{fact}\label{fa:kontr}
 Consider a set $T\subset \eR^n$ and $x,y\in \eR^n$ such that for any $i\leq n$, $|x_i|\leq |y_i|$. Then
 \begin{align*}
 \Ex \sup_{t\in T} \sum_i t_i x_i \eps_i \leq \Ex \sup_{t\in T} \sum_i t_i y_i \eps_i, \textrm{ and }\lv \sum_i x_i \eps_i \rv_p \leq \lv \sum_i y_i \eps_i \rv_p.
 \end{align*}

\end{fact}
\begin{proof}
    Since $\varphi_i(t)=t \cdot \1_{y_i \neq 0}  \frac{x_i}{y_i}$ is a contraction ($\lv \varphi_i \rv_{Lip} \leq 1$), the first inequality follows from \cite[Theorem 6.5.1]{tal}. The second one is an easy consequence of \cite[Theorem 4.4]{proinban}.
\end{proof}
Using standard arguments, we can derive a counterpart of the contraction principle for general symmetric random variables.
\begin{lemma}\label{lem:porsup}
    Let $X_1,\ldots,X_n$ and $Y_1,\ldots,Y_n$ be two sequences of independent, symmetric random variables such that, for a certain $M>0$
    \[\forall_{t\geq M} \forall_{i\leq n} \, \Pro(|X_i|\geq t) \leq \Pro(|Y_i| \geq t).\]
    Then for any set $T\subset \eR^n$
    \[\Ex \sup_{t\in T} \sum_i t_i X_i \leq  \left(1+\frac{M}{\min_i \Ex |Y_i|}\right) \Ex \sup_{t\in T} \sum_i t_i Y_i,\]
    and for any $a_1,\ldots,a_n \in \eR$, $p\geq 1$
    \[ \lv \sum_i a_i X_i \rv_p \leq\left(1+\frac{M}{\min_i \Ex |Y_i|}\right)\lv \sum_i a_i Y_i \rv_p. \]
\end{lemma}
\begin{proof}
  By inversing the CDF we may assume that $X_i$, $Y_i$ are defined on the same probability space and that
    \[|X_i|\1_{|X_i|\geq M}\leq |Y_i|\1_{|Y_i|\geq M} \ a.s..\]
  Let $\eps_1,\ldots,\eps_n$ be independent of $(X_i)_{i\leq n},(Y_i)_{i\leq n}$. Using Fact \ref{fa:kontr} conditionally on $(X_i),(Y_i)$ we get that
    \begin{align*}
        \Ex \sup_{t\in T}\sum_i t_i X_i\1_{|X_i|\geq M}&= \Ex \sup_{t\in T}\sum_i t_i |X_i|\1_{|X_i|\geq M} \eps_i \\
        &\leq \Ex \sup_{t\in T} \sum_i t_i |Y_i|\eps_i=\Ex \sup_{t\in T} \sum_i t_i Y_i.
    \end{align*}
Using the same argument and Jensen's inequality,
    \begin{align*}
      \Ex \sup_{t\in T}\sum_i t_i X_i\1_{|X_i|< M} &\leq M \Ex \sup_{t\in T} \sum_i t_i \eps_i \\
      &\leq \frac{M}{\min_i \Ex |Y_i|} \Ex \sup_{t\in T} \sum_i t_i \eps_i \Ex |Y_i| \leq \frac{M}{\min_i \Ex |Y_i|} \Ex \sup_{t\in T} \sum_i t_i Y_i.   
    \end{align*}
        Thus, the first inequality holds. We show the second in the same manner.
\end{proof}

The following is a direct consequence of Lemma \ref{lem:porsup}.
\begin{lemma}\label{lem:zmnienwspol}
  Let  $X_1,\ldots,X_n$ be independent, symmetric random variables. Consider real numbers  $(a_i)_{i\leq n},(b_i)_{i\leq n}$  such that $|a_i|\leq |b_i|$ for $i\leq n$. Then
  \[ \lv \sum_i a_i X_i \rv_p \leq \lv \sum_i b_i X_i \rv_p.\]
\end{lemma}

The next lemma is a straightforward application of Jensen's inequality. However, since the argument is used several times in subsequent work, we decided to state it as a separate fact.

\begin{lemma}\label{lem:turbojensen}
Let $X_1,\ldots,X_n,Y_1,\ldots,Y_n$ be independent, symmetric random variables. Then for any real numbers $a_1,\ldots,a_n$ and any $p\geq 1$
\[\lv \sum_i a_i X_iY_i \rv_p \geq \lv \sum_i a_i X_i \Ex |Y_i| \rv_p. \]
\end{lemma}
\begin{proof}
    Let $(\eps_i)$ be independent of $(X_i),(Y_i)$. Since the $(X_i),(Y_i)$ are symmetric, Jensen's inequality implies that
    \begin{align*}
     \lv \sum_i a_i X_iY_i \rv_p &=\lv \sum_i a_i X_i\eps_i |Y_i|  \rv_p \geq  \lv \sum_i a_i X_i \eps_i  \Ex |Y_i|  \rv_p =\lv \sum_i a_i X_i \Ex |Y_i| \rv_p.
    \end{align*}
\end{proof}



We will need the following result regarding the number of connected subsets of a graph.
\begin{lemma}\cite[Lemma 3.2]{latop}\label{lem:grafzlicz}
    Let $H=(V,E)$ be a graph with $m$ vertices and maximum degree $d_H$. Then, the number of connected, subsets $I\subset V$ of cardinality $k$ is at most $m(4d_H)^{k-1}$.
\end{lemma}
The following general fact is surprisingly useful.
\begin{lemma}\label{lem:oszmaxpodex}
   Let $(X_{ij})_{i,j \leq n}$ be random variables (any). Then 
   \[\Ex \max_{i,j \leq n} |X_{ij}| \leq e^2 \max_{ij} \lv X_{ij} \rv_{\Log \, n} \leq e^2 \sup_{v,w \in B_2} \lv \sum_{ij} v_i w_j X_{ij} \rv_{\Log \, n}=e^2R_X(A) .\]
\end{lemma}
\begin{proof}
    Using Jensen's inequality we obtain
    \begin{align*}
       \Ex \max_{i,j \leq n} |X_{ij}| \leq \Ex \sqrt[\Log \, n]{\sum_{i,j \leq n} |X_{ij}|^{\Log \, n}}\leq  \sqrt[\Log \, n]{\Ex \sum_{i,j \leq n} |X_{ij}|^{\Log \, n}}\leq \sqrt[\Log \, n]{n^2} \max_{ij} \lv X_{ij} \rv_{\Log \, n}. 
    \end{align*}
\end{proof}

\begin{remark}\label{rem:da}
 By Lemma \ref{lem:oszmaxpodex}
\[ \Ex \sup_{v,w\in B_2} \sum_{i} a_{ii}X_{ii}v_i w_i=\Ex \max_i |a_{ii}X_{ii}| \leq R_X(A). \]
Therefore, it suffices to prove Theorem \ref{tw:glow} for matrices with a zero diagonal. 
\end{remark}
We conclude this section with a deeper structural results concerning random variables from the \Slev class. We begin by quoting the result which state that the classical decomposition theorem
for Gaussian processes can be generalized. This was first observed by Adamczak and Lata{\l}a in a slightly less general case (cf.~\cite[Lemma~5.10]{adlat}).
\begin{lemma}\cite[Lemma 5.4]{d2}\label{lem:rozbij}
 Assume that $X_1,\ldots,X_n$ are independent, \Slev-random variables. Then for any sets $T_1,\ldots,T_m \subset \eR^n$
 \[\Ex \sup_{t\in \bigcup_{l=1}^m}\sum_i t_i X_i \lesssim^\kappa \max_{l\leq n} \Ex \sup_{t\in T_k} \sum_i t_i X_i + \sup_{s,t\in \bigcup_{l=1}^m T_l} \lv \sum_i (t_i-s_i)X_i \rv_{\Log \, m}.\]
\end{lemma}
\begin{theorem}\cite[Theorem 4.1]{d2}\label{tw:formulamomenty}
    Assume that $(X_{ij})_{i,j \leq n}$ are independent random variables from the \Slev class such that $\Ex |X_{ij}|=1$ for $i,j=1,\ldots,n$. Then for any matrix $(a_{ij})_{i,j \leq n}$
    \[\lv \sum_{ij} a_{ij} X_{ij} \rv_p \rown^\kappa \sup_{t\in B^X_p} \sum_{ij} a_{ij} t_{ij},  \]
    where we recall the definition of $B^X_p$ given in \eqref{def:BPX}.
\end{theorem}
It is striking that in order to prove Theorem \ref{tw:main} we do not need the precise definition of the set $B^X_p$. It is enough to know that for any fixed $(X_{ij})_{i,j \leq n}$ from the \Slev class, there exists $1$-unconditional set $T=T(X)\subset \eR^{n^2}$ such that
\begin{equation}\label{eq:wlasnoscT}
      \sup_{t\in T} \left| \sum_{ij} t_{ij} a_{ij} \right|=  \sup_{t\in T} \sum_{ij} \left|  t_{ij} a_{ij} \right| =  \sup_{t\in T} \sum_{ij} t_{ij} a_{ij} \rown^\wym \lv \sum_{ij} a_{ij} X_{ij} \rv_{\Log \, n}.
\end{equation}
We recall that a set $T\subset \eR^{n^2}$ is called $1$-unconditional if, for any $t\in T$ and any choice of $x_{ij}=\pm 1$, we have $(t_{ij} x_{ij})_{i.j\leq n} \in T$ (clearly $B^X_p$ is $1$-unconditional). In particular if $T$ is the set from \eqref{eq:wlasnoscT} and $(Y_{ij})_{i,j \leq n}$ are random variables which are independent of $(X_{ij})_{i,j \leq n}$ then
\begin{align}\label{eq:suptpodmomentem}
    \lv \sup_{t\in T} \sum_{ij}a_{ij}t_{ij}Y_{ij}\rv_{\Log \, n} \rown^\wym \lv  \sqrt[\Log \, n]{ \Ex^X \left| \sum_{ij}a_{ij}t_{ij} X_{ij}Y_{ij}\right|^{\Log \, n} }   \rv_{\Log \, n}=\lv   \sum_{ij}a_{ij}t_{ij} X_{ij}Y_{ij}  \rv_{\Log \, n}.
\end{align}

\section{Proofs of Theorems \ref{tw:main} and \ref{twr:oszacujRX}} \label{sec:tw1.1}

We start with the fact that the correct asymptotic behavior of moments implies the existence of sub-exponential moments. The proof is straightforward, but we provide it for the reader's convenience.
\begin{fact}\label{fa:oszmomwyk}
Assume that $X$ is a random variable, such that there exist $C_1,r>0$ such that $\lv X \rv_p \leq C_1p^r$ for all $p\geq 1$. Define $\eta=\eta(C_1,r):=\min(1,\frac{r}{2C_1^{1/r}e})$. Then  $\Ex \exp(\eta |X|^{1/r})\leq C_1e+5/e$.
\end{fact}
\begin{proof}
If $p\leq 1$ then $\Ex |X|^p \leq \Ex |X| \leq C_1$ so we have
    \begin{align*}
     \Ex \exp(\eta|X|^{1/r})-1&=\sum_{k=1}^\infty \frac{\eta^k \Ex |X|^{k/r}}{k!} \leq \sum_{k\leq r} \frac{\eta^k C_1}{k!}+\sum_{k>r}\frac{\eta^kC_1^{k/r}\left( \frac{k}{r}\right)^{r \frac{k}{r}}}{k!} \\
     &\leq C_1e+\sum_{k>r}\frac{\eta^kC_1^{k/r}\left( \frac{k}{r}\right)^{k}}{k!}. 
    \end{align*}
    Since $k! \geq e(k/e)^k$, we can bound the latter sum by
    \begin{align*}
        \sum_{k>r}\frac{\eta^kC^{k/r}\left( \frac{k}{r}\right)^{k}}{k!} \leq \sum_{k>r}\left(\frac{\eta C_1^{1/r}e}{r} \right)^k \frac{1}{e}\leq \sum_{k=0}^\infty \frac{2^{-k}}{e}=2/e.
    \end{align*}
\end{proof}

\begin{fact}\label{fa:oszog}
Let $X$ be from \Slev class and assume that $\Ex |X|\leq 1$. Then, there exist $\eta(\kappa),C(\kappa)>0$ such that, for $t>C(\kappa)$, we have that  
\[\Pro(|X|\geq t) \leq e^{-\eta(\kappa)t^{\frac{1}{\parr}}}.\]
\end{fact}
\begin{proof}
The definition of the \Slev class implies that, for any $p\geq 1$ (we recall that $\kappa \geq 2$)
\[\lv X \rv_p \leq \lv X \rv_{2^{\lceil \log_2 p\rceil}} \leq \kappa^{\lceil \log_2 p\rceil} \leq \kappa \kappa^{\log_2 p}=\kappa p^{\parr}.  \]
Thus, $X$ satisfies the assumptions of Fact \ref{fa:oszmomwyk} with $r=\parr$ and $C_1=\kappa$. Let $\eta=\eta(\kappa)$ be as in Fact \ref{fa:oszmomwyk}. Markov's inequality implies that
    \begin{align*}
        \Pro(|X|\geq t) &= \Pro\left(e^{\eta |X|^{\odpar }}  \geq e^{\eta |t|^{\odpar }} \right) \\
        &\leq e^{-\eta |t|^{\odpar }} \Ex  e^{\eta |X|^{\odpar }} \leq e^{-\eta |t|^{\odpar }} (\kappa e+5/e) \leq e^{-\frac{\eta}{2} |t|^{\odpar }},
    \end{align*}
    where the last inequality holds for sufficiently large $t$.
\end{proof}
We say that $X$ has a Weibull$(r)$ distribution ($r>0$) if $X$ is a symmetric random variable and $\Pro(|X|\geq t)=e^{-t^{r}}$.
\begin{corollary}\label{cor:porsup}
Let $X_1,\ldots,X_n$ be  independent random variables from the \Slev class such that $\Ex |X_i|=1$ for each $i\leq n$. Consider $Y_1,\ldots,Y_n$ which are independent Weibull($\odpar$) random variables. Then
\[\Ex \sup_{t\in T} \sum_i t_i X_i \leq C(\kappa) \Ex \sup_{t\in T} \sum_i t_i Y_i.\]
\end{corollary}
\begin{proof}
    This is an easy consequence of Lemma \ref{lem:porsup} and Fact \ref{fa:oszog}.
\end{proof}

The following theorem follows directly from \cite[Theorem~4.4]{lathan}.
\begin{theorem}\label{twr:opweib}
Let $(Y_{ij})_{i,j \leq n}$ be independent Weibull$(r)$ random variables, where $r\leq 2$. Then, for any symmetric matrix $(a_{ij})_{i,j,\leq n}$
\[\Ex \lv (a_{ij} Y_{ij})_{i,j \leq n} \rv_{op} \leq \max_i \sqrt{\sum_j a^2_{ij}}+\Log^{1/r}(n) \max_{ij} |a_{ij}|.\]
\end{theorem}

\begin{proof}[Proof of Theorem \ref{tw:main}]
The proof of \cite[Remark 4.5]{latswiat}, based on the permutation method from \cite{lathan} shows that if for any matrix $A$ we have
    \[\Ex \lv (a_{ij} X_{ij})_{i,j \leq n} \rv_{op} \leq \beta(\kappa,n)\left(M(A)+R_X(A)\right), \]
    then a stronger inequality holds
    \[\Ex \lv (a_{ij} X_{ij})_{i,j \leq n} \rv_{op} \leq C\beta(\kappa,n) \left(M(A)+D(A,X) \right).\]
    The constant $\beta$ is equal in both above inequalities. In other words, the cost of applying the permutation trick is the appearance of a multiplicative numerical constant. Hence \eqref{eq:twglowneoszacowaniepodwojnylog} implies \eqref{eq:wstepdwustronne}. Moreover, the lower bound in \eqref{eq:twglowneoszacowaniepodwojnylog} was proved in \cite{latswiat} (Theorem 4.1 therein). Thus, we will prove only the upper bound.\\
First, consider a symmetric matrix $A=(a_{ij})_{i,j \leq n}$. In particular
\[M=M(A)=2 \mia.\]
Let $(Y_{ij})_{i,j\leq n}$ be independent random variables with Weibull$(r)$ distribution, where $r:=\frac{1}{\log_2 \kappa}$ (recall that $\kappa\geq 2$). By Corollary \ref{cor:porsup} and Theorem \ref{twr:opweib}
    \begin{align*}
     \Ex \lv (a_{ij}\1_{|a_{ij}|\leq \frac{M}{\Log^r(n)}}X_{ij})_{i,j \leq n} \rv_{op} \lesssim^\kappa \Ex \lv  (a_{ij} \1_{|a_{ij}|\leq \frac{M}{\Log^r(n)}} Y_{ij})_{i,j \leq n} \rv_{op} \lesssim^\kappa  M.   
    \end{align*}
    On the other hand
    \[\max_i \left|\left\{j: |a_{ij}|>\frac{M}{\Log^r\, n} \right\} \right| \leq \Log^{2r}\, n. \]
    Thus if $\hat{A}:=(a_{ij}\1_{|a_{ij}|>\frac{M}{\Log^r \, n} })_{i,j \leq n}$ then $d_{\hat{A}}\leq \Log^{2r} \, n$. Since by Lemma \ref{lem:zmnienwspol}  $R_X(\hat{A}) \leq \RX$ formula \eqref{eq:twglownezda} implies that
    \begin{align*}
        \Ex \lv (a_{ij}\1_{|a_{ij}|>\frac{M}{\Log^r \, n} }X_{ij})_{i,j \leq n} \rv_{op} \lesssim^\kappa \Log^{C(\kappa)}(\Log^{2r} \, n)\left(M+\RX\right).
    \end{align*}
Thus, Theorem \ref{tw:glow} holds for symmetric matrices since,
    \[\Log^{C(\kappa)}(\Log^{2r} \, n)\lesssim^\kappa 1+\Log^{C(\kappa)}(\Log \, n) \lesssim^\kappa \Log^{C(\kappa)} (\Log \, n). \]
Now, consider any matrix $A=(a_{ij})_{i,j\leq n}$. Let $(X_{ij})_{i,j}, (X'_{ij})_{i,j} \leq n$ be independent random variables that satisfy assumptions of Theorem \ref{tw:glow}. We also assume that $X'_{ji}$ is distributed as  $X_{ij}$ (the transposition of indices is intentional).  Consider
\[A^{sym}=(a^{sym}_{ij})_{i,j \leq 2n}:=\begin{bmatrix}
    0 &A\\
    A^T &0
\end{bmatrix}.\]
 Since $A^{sym}$ is a symmetric, block matrix (by slight abuse of notation) Theorem \ref{tw:glow} for symmetric matrices implies
\begin{align*}
     \Ex  \lv (a_{ij}X_{ij})_{i,j \leq n} \rv_{op} &\leq \Ex \lv \begin{bmatrix} 0 &(a_{ij}X_{ij})_{i,j \leq n} \\ ((a_{ij}X'_{ij})_{i,j \leq n})^T &0 \end{bmatrix}\rv_{op}\\
       &\lesssim^\kappa \Log^{C(\kappa)} \Log(n) \left(\max_i \lv (a^{sym}_{ij})_j \rv_2+R_{(X,X')}(A^{sym})\right). 
\end{align*}
Clearly 
 \[\max_i \lv (a^{sym}_{ij})_j \rv_2=\max_j \lv (a^{sym}_{ij})_i \rv_2=\max \left(\mia, \mja\right),\]
 and (again by slight abuse of notation)
 \begin{align*}
R_{(X,X')}(A^{sym})&=\sup_{v,w\in B_2} \lv v \cdot \begin{bmatrix} 0 &A\ast X \\ (A\ast X')^T &0 \end{bmatrix} w^T \rv_{\Log \, n} \\
&\leq 2 \sup_{v,w \in B_2} \lv \sum_{ij} a_{ij}X_{ij}\rv_{\Log \, n}=2\RX. 
 \end{align*}
\end{proof}
\begin{proof}[Proof of Theorem \ref{twr:oszacujRX}]

    Fix $I\subset [n]\times [n]$ with cardinality $p$. For any $v,w \in B_2$ by Jensen's inequality and Theorem \ref{tw:formulamomenty} 
    \[\lv \sum_{ij} a_{ij}v_iw_j X_{ij} \rv_p \geq \lv \sum_{(i,j)\in I} a_{ij}v_iw_j X_{ij} \rv_p \gtrsim^\kappa \sup_{t\in B^X_p} \sum_{(i,j)\in I} a_{ij}v_iw_j t_{ij}. \]
    The lower bound follows by taking supremum over $v,w\in B_2$. \\
    To show the upper bound we fix  $v,w\in B_2$ and apply Theorem \ref{tw:formulamomenty}
    \begin{align*}
     \lv \sum_{ij} a_{ij}v_iw_j X_{ij} \rv_p  &\approx^\kappa \sup_{t\in B^X_p} \sum_{ij} a_{ij}v_iw_j t_{ij} \\
     &\leq \sup_{t\in B^X_p} \sum_{ij} a_{ij}v_iw_j t_{ij}\1_{|t_{ij}|\leq 1}+\sup_{t\in B^X_p} \sum_{ij} a_{ij}v_iw_j t_{ij}\1_{|t_{ij}|> 1}.    
    \end{align*}
    For any $t> 1$ Markow's inequality implies (we recall our normalization $\Ex |X_{ij}|=1$)
    \[\hat{N}^X_{ij}(t)=-\ln\Pro(|X_{ij}|\geq t) \leq -\ln \Ex |X_{ij}| = 1. \]
    This implies that
    \begin{equation}\label{eq:locdwarf}
    \sup_{t\in B^X_p} \sum_{ij}\1_{|t_{ij}|> 1} \leq p.    
    \end{equation}
The above inequality is not necessarily true under the normalization $\Ex |X_{ij}| = 1$. This is the reason why we temporarily changed the normalization.
     Formula \eqref{eq:locdwarf} yields
    \[\sup_{t\in B^X_p} \sum_{ij} a_{ij}v_iw_j t_{ij}\1_{|t_{ij}|> 1} \leq \sup_{\substack{I\subset [n]\times [n]\\ |I|=\Log n}} \sup_{ t\in B^X_p} \lv (a_{ij}t_{ij})_{(i,j)\in I} \rv_{op}.\]
    Now let $(\eps_{ij})_{i,j \leq n}$ be independent Rademacher random variables (symmetric $\pm 1$ random variables).
    Since
    \[\hat{N}^\eps_{ij}(t)=t^2\1_{|t| \leq 1},\]
Theorem \ref{tw:formulamomenty} implies that (Rademacher random variables satisfies its assumptions with $\kappa$=1)
    \[ \sup_{t\in B^X_p} \sum_{ij} a_{ij}v_iw_j t_{ij}\1_{|t_{ij}|\leq 1}=\sup_{\lv t \rv_2 \leq p, \lv t \rv_\infty \leq 1} \sum_{ij} a_{ij}v_iw_j t_{ij} \lesssim \lv (a_{ij} v_i w_j \eps_{ij})_{i,j \leq n} \rv_{p}.\]
By \cite[Proposition 1.5]{latswiat}
\begin{align*}
 \lv (a_{ij} v_i w_j \eps_{ij})_{i,j \leq n} \rv_{p} &\lesssim \sqrt{\Log \, p} \sup_{\substack{I\subset [n]\times [n]\\ |I|= p}}\lv (a_{ij})_{(i,j)\in I} \rv_{op}\\
 &\leq \sqrt{\Log \, p}\sup_{\substack{I\subset [n]\times [n]\\ |I|=p}} \sup_{ t\in B^X_p} \lv (a_{ij}t_{ij})_{(i,j)\in I} \rv_{op}.    
\end{align*}
The latter inequality holds, since for any $|I|=p$, we have $(\1_{(i,j)\in I})\in B^X_p.$
\end{proof}

\section{The subgaussian case}\label{sec:subgaus}
In this section, we prove a version of Theorem~\ref{tw:glow} for subgaussian
random variables, where there is an additional supremum over (some) set
$T \subset \eR^{n^2}$ (see Theorem~\ref{prop:oszapodgaus}). The presence of the additional set $T$ under the supremum will allow us to
carry out an inductive proof with respect to $\sqrt{\log_2 \kappa}$, where $\kappa$ is the constant in~\eqref{eq:*}. A random variable $X$ is subgaussian if  $\Ex X=0$ and if there exists a constant $K_1\leq \infty$ such that
\begin{equation*}
 \lv X \rv_{\psi_2}:=\sup\{t>0: \Ex e^{(tX)^2} \leq 2 \}=K_1<\infty.   
\end{equation*}
In this case, we say that $X$ is $K_1$-subgaussian. We recall the following well-known theorem.
\begin{theorem}\label{thm:podgausschar}
Let $X$ be a random variable such that $\Ex X=0$. The following conditions are equivalent:
\begin{enumerate}
    \item\label{war:1} The random variable $X$ is $K_1$-subgaussian. 
    \item Exists $K_2<\infty$ such that for all $p \geq 1$ we have $\lv X \rv_p \leq K_2 \sqrt{p}\, \Ex |X_1|$.
    \item Exists $K_3<\infty$ such that for all $q \geq p \geq 1$ we have  $\lv X \rv_q \leq K_3 \sqrt{q/p} \lv X \rv_p$.
    \item\label{war:2} Exists $K_4<\infty$ such that for any real $\lambda$ we have $\Ex \exp(\lambda X) \leq 2 \exp(K^2_4 \lambda^2)$.
\end{enumerate}
Furthermore, if any of these cases holds true, then the constants can be chosen such that $K_1 \rown K_2 \rown K_3 \rown K_4$. 
\end{theorem}

For the remainder of this section, we will assume that the random variables $(X_{ij})$
are independent, symmetric, and $\eta$-subgaussian. The theorem stated below is a consequence of the result above
and follows specifically from the equivalence between conditions~\ref{war:1} and~\ref{war:2}.
Moreover, the fact that the variables $(X_{ij})$ belong to the class
$\mathrm{SRV}(C\eta)$ is, in turn, an immediate consequence of the theorem
stated below. In particular, they satisfy the assumptions of Lemma~\ref{lem:rozbij}.

\begin{fact}\label{fa:sumujpodgaussy}
 Assume that $X_1,\ldots,X_n$ are independent $\eta$-subgaussian random variables. Then $X=\sum_{i\leq n} a_i X_i$ is $C\eta \sqrt{\sum_i a^2_i}$-subgaussian. In particular
 \[\lv \sum_i a_i X_i \rv_q \lesssim \eta \sqrt{q/p} \lv \sum_i a_i X_i \rv_p.\]
\end{fact}

As mentioned in the introduction, this section repeats Lata{\l}a's argument from~\cite{latop}.
Accordingly, we adopt the notation used in that work.\\
For the remainder of this section, we fix a symmetric matrix
$A = (a_{ij})_{i,j \leq n}$ and an arbitrary set $T \subset \eR^{n^2}$ (the one mentioned at the beginning of the section). By Remark \ref{rem:da} we may and will assume that $a_{ii}\equiv 0$. We recall that $G_A([n],E_A)$ is a graph such that $(i,j)\in E_A$ if and only if $i\neq j$ and $a_{i,j}\neq 0$, and $d_A$ is  the degree of $G_A$. We define
\begin{itemize}
    \item $\rho=\rho_A$ -  the distance on $[n]$ induced by $E_A$,
    \item $G_r=G_r(A)=([n],E_{A,r})$, $r=1,\ldots,n$ - graphs such that $(i,j)\in E_{A,r}$ iff $\rho(i,j)\leq r$.
\end{itemize}
Clearly $G_1=([n],E_A)$. Moreover, the maximum degree of $G_r$ is at most $d_A+d_A(d_A-1)+\ldots+d_A(d_A-1)^{r-1}\leq d^r_A$. We say that a subset of $[n]$ is $r$-connected if it is connected in $G_r$. \\
We define further
\begin{itemize}
    \item  $\Ik=\mathcal{I}(k,n)$ - the family of all subsets of $[n]$ of cardinality $k$,
    \item $\mathcal{I}_r(k)=\mathcal{I}_r(k,A)$ - the family of all $r$-connected subsets of $[n]$ of cardinality $k$,
    \item $I'=I'(A)$ -  the set of all neighbors of $I$ in $G_1$,
    \item $I''=I''(A)$ - the set of all neighbors of $I'$ in $G_1$.
\end{itemize}
Thus, by definition
\begin{align*}
    &I'=\{j\in [n]:\exists_{i\in I}\, (i,j)\in E_A\}, & &I''=\{ i\in [n]: \exists_{i_0\in I,j\in [n]}\, (i_0,j),(i,j)\in E_A\}.
\end{align*}
Naturally, $I\subset I''$ but in general $I \subsetneq I'$. Moreover $|I'|\leq d_A|I|$ and $|I''|\leq d^2_A|I|$. For a set $I\subset [n]$ and a vertex $j\in [n]$ we write $I \sim_A j$ if $(i,j)\in E_A$ for some $i\in I$.\\
For $k,l=1,\ldots,n$ we define random variables
\[\XX_{kl}=\XX_{kl}(A,T,X):=\frac{1}{\sqrt{kl}} \max_{I\in \Ik, J\in \Il} \sup_{t\in T} \max_{\eta_i,\eta'_j=\pm 1}\sum_{ij}a_{ij}X_{ij}t_{ij}\eta_i \eta'_j.  \]
Set also
\begin{align}
\XX=\XX(A,T,Y):=\max_{1 \leq k,l \leq n} \XX_{k,l}=\max_{\emptyset \neq I,J \subset [n]} \frac{1}{\sqrt{|I||J|}} \max_{\eta_i,\eta_j'=\pm 1} \sup_{t\in T} \sum_{i\in I,j\in J} a_{ij} X_{ij}\eta_i \eta'_j t_{ij}.\label{eq:defxx}
\end{align}
Since $A$ is a symmetric matrix, the parameter $d_A$ controls the number of nonzero
elements among $(a_{ij})_{i \le n}$ (with $j$ fixed) and $(a_{ij})_{j \le n}$
(with $i$ fixed). This allows us to control the number of nonzero terms
in the above sum. The information provided by $d_A$ is sufficient for our purposes. We do not exploit any additional properties of the set $T$. Consequently,
no assumptions on $T$ are required (such as symmetry). We also define the $4$-connected counterpart of $\XX_{kl}$.
\[\XXb_{kl}=\XX_{k,l}(A,T,X):=\frac{1}{\sqrt{kl}} \max_{I\in \Ikc, J\in \Ilc} \sup_{t\in T} \max_{\eta_i,\eta'_j=\pm 1}\sum_{ij}a_{ij}X_{ij}t_{ij}\eta_i \eta'_j.   \]
 The random variables $\XXb_{kl}$ are easier to upper bound than $\XX_{kl}$ because the number of $4$-connected subsets is much smaller than the total number of subsets. However, the next lemma shows that $(\XXb_{kl})_{k,l\leq n}$ controls $(\XX_{kl})_{k,l\leq n}$. First, observe that, in our notation
 \begin{align*}
     \sup_{t\in T, v,w\in B_2} \lv \sum_{ij} a_{ij} t_{ij} X_{ij} v_i w_j \rv_{\Log \, n}=\RTX.
 \end{align*}
 
\begin{lemma}\label{lem:pierwszy}
 It is true that
    \[\Ex \XX = \Ex \max_{k,l} \XX_{k,l} \lesssim^\eta \max_{1\leq k',l' \leq n} \Ex \XXb_{k',l'}+\RTX.  \]
\end{lemma}
\begin{proof}
   Fix $I\in \Ik, J\in \Il$. Let $I_1,\ldots,I_r$ be the connected components of $I\cap J'$ in $G_2$, and let $J_u:=J\cap I'_u$. The sets $J_1,\ldots,J_r$ are disjoint. 
  They are also $4$-connected subsets of $J$; otherwise there would exist a nonempty set $V \subsetneq J_u$ such that
$\rho_A(V, J_u \setminus V) \ge 4$. Let $\tilde{V}$ be the set of neighbors of $V$ in $I_u$. Then $\emptyset \neq \tilde{V} \subsetneq I_u$ and
$\rho_A(\tilde{V}, I_u \setminus \tilde{V}) \ge 2$, which contradicts the $2$-connectivity of $I_u$. Hence, for every  $t\in T$ and $\eta_i,\eta_j'=\pm 1$ we have
   \begin{multline*}
       \sum_{i\in I,j\in J} t_{ij}a_{ij}X_{ij}\eta_i \eta_j'=\sum_{i\in I\cap J',j\in J} t_{ij}a_{ij}X_{ij}\eta_i\eta'_j=\sum_{u=1}^r \sum_{i\in I_u,j\in J_u}t_{ij}a_{ij}X_{ij}\eta_i \eta_j' \\
       \leq \sum_{u=1}^r \XXb_{|I_u|,|J_u|} \sqrt{|I_u||J_u|}\leq \max_{k',l'}\XXb_{k',l'}\sum_{u=1}^r \sqrt{|I_u||J_u|}\\
       \leq \max_{k',l'}\XXb_{k',l'} \left(\sum_{u=1}^r |I_u| \right)^{1/2}\left(\sum_{u=1}^r |J_u| \right)^{1/2}=\max_{k',l'}\XXb_{k',l'}\sqrt{|I||J|}.
   \end{multline*}
   Taking the supremum over all $k,l \leq n$, sets $I\in \Ik$, $J\in \Il$, signs $\eta_i,\eta'_j =\pm 1 $ and $t\in T$ yields
   \begin{equation*}
     \XX_{k,l} \leq \max_{k',l'}\XXb_{k',l'}.
   \end{equation*}
   Since the right side does not depend on $k,l$ we have
   \begin{equation}\label{eq:loc1o}
     \Ex \XX =\Ex \max_{k,l} \XX_{k,l} \leq \Ex \max_{k,l} \XXb_{k,l}.  
   \end{equation}
   
Fact \ref{fa:sumujpodgaussy} implies that
   \begin{align*}
    &\max_{k,l}\max_{I\in \Ikc,J\in \Ilc} \max_{\eta_i,\eta_j'=\pm 1} \sup_{t\in T} \frac{1}{\sqrt{kl}} \lv \sum_{i\in I,j\in J} a_{ij}X_{ij} t_{ij} \eta_i \eta'_j \rv_{\Log \, n^2}\\
    &\leq  \sup_{t\in T, v,w\in B_2} \lv \sum_{ij} a_{ij} t_{ij} X_{ij} v_i w_j \rv_{2\Log \, n} \lesssim^\eta \RTX .
   \end{align*}
By Lemma \ref{lem:rozbij} we have that (the random variables $(X_{ij})_{i,j\leq n}$ belong to the $SRV(C\eta)$ class)
\[\Ex \max_{k',l'}\XXb_{k',l'} \lesssim^\eta \max_{k',l'} \Ex\XXb_{k',l'}+\RTX . \]
The assertions follows by \eqref{eq:loc1o} and the preceding inequality.
\end{proof}

\begin{lemma}\label{lem:drugiwstep}
   For any $1,\leq k,l\leq n$, we have
    \begin{align*}
\Ex \XXb_{k,l} \leq \Ex \sup_{t\in T} \lv (a_{ij}X_{ij}t_{ij} \rv_{op} \lesssim^\eta d_A \Ex \sup_{t\in T} \max_{ij}|a_{ij}X_{ij}t_{ij}|.
    \end{align*}
\end{lemma}
\begin{proof}
Fix $t\in T$ and observe that for any $v,w\in B_2$
\begin{align*}
\left|\sum_{i\neq j} a_{ij}X_{ij}t_{ij}v_iw_j \right|&\leq \max_{ij}|a_{ij}X_{ij}t_{ij}| \sum_{ij} \1_{a_{ij}\neq 0} \frac{v^2_i+w^2_j}{2} \\
&=\frac{\max_{ij}|a_{ij}X_{ij}t_{ij}|}{2}\left( \sum_{i} v^2_i \sum_j \1_{a_{ij}\neq 0}+ \sum_{j} w^2_j \sum_i \1_{a_{ij}\neq 0} \right) \\
&\leq d_A \max_{ij}|a_{ij}X_{ij}t_{ij}|.
\end{align*}
The assertion follows by taking the supremum over $t \in T$ and expectations. 
\end{proof}

\begin{lemma}\label{lem:drugi}
   For any $1\leq k,l\leq n$, we have
    \begin{multline*}
\Ex \XXb_{k,l} \lesssim^\eta \sup_{v,w \in B_2}\Ex \sup_{t\in T}\sum_{ij}a_{ij} X_{ij} t_{ij}v_iw_j + \RTX +\Ex \max_{ij} \sup_{t\in T} \left| a_{ij} X_{ij} t_{ij}\right|\\
+\sqrt{\Log \, d_A} \left(\MTI+\MTJ\right).
    \end{multline*}
\end{lemma}
\begin{proof}
If $d_A\leq 2$ then the assertion follows from Lemma \ref{lem:drugiwstep}.  Assume that $d_A\geq 3$. Clearly
\begin{align*}
  \max_{I\in \Ikc,\, J\in \Ilc} \max_{\eta_i,\eta'_j=\pm 1} \frac{1}{\sqrt{kl}}\Ex \sup_{t\in T} \sum_{i\in I} \sum_{j\in J} a_{ij}t_{ij} X_{ij} \eta_i \eta'_j \leq \sup_{v,w \in B_2}\Ex \sup_{t\in T}\sum_{ij}a_{ij} X_{ij} t_{ij}v_iw_j .
\end{align*}
By Lemma   \ref{lem:grafzlicz}, we have that $2^k| \Ikc| \leq n (8d^4_A)^k\leq nd_A^{6k}$. Similarly, $2^l |\Ilc| \leq n d_A^{6l}$. Observe that
    \[\Log(n^2 d_A^{6(k+l)}) \lesssim \begin{cases}
       \Log \, n &n^2 \geq (d_A)^{6(k+l)} \\
       (k+l) \Log \, d_A &\textrm{otherwise.}
    \end{cases}\]
In the first case (using Fact \ref{fa:sumujpodgaussy})
\begin{multline*}
  \max_{I\in \Ikc,J\in \Ilc} \max_{\eta_i,\eta'_j=\pm 1} \sup_{t\in T}\frac{1}{\sqrt{kl}}\lv \sum_{i\in I} \sum_{j\in J} a_{ij}t_{ij} X_{ij} \eta_i \eta'_j  \rv_{\Log(n^2 d_A^{6(k+l)})}\\
  \lesssim^\eta \RTX.  
\end{multline*}
Since $X_{ij}$ is an $\eta$-subgaussian random variable that satisfies $\Ex |X_{ij}|=1$ (the normalization), we have that
\[\Ex X^2_{ij}\leq C^2\eta^2 \left(\Ex |X_{ij}|\right)^2=C\eta^2.\]
Thus, in the second case (i.e. $n^2<(d_A)^{6(k+l)}$),  Fact \ref{fa:sumujpodgaussy} implies that, for  $k\leq l$ (we bound the higher moment by the second moment)
\begin{align*}
    \max_{I\in \Ikc,\,J\in \Ilc} \max_{\eta_i,\eta'_j=\pm 1} \sup_{t\in T}\frac{1}{\sqrt{kl}}\lv \sum_{i\in I} \sum_{j\in J} a_{ij}t_{ij} X_{ij} \eta_i \eta'_j  \rv_{\Log(n^2 d_A^{6(k+l)})}\\
    \lesssim^\eta \max_{I\in \Ikc,\,J\in \Ilc}  \sup_{t\in T}\frac{\sqrt{(k+l) \Log(d_A)}}{\sqrt{kl}} \sqrt{\sum_{i\in I} \sum_{j\in J} a^2_{ij}t^2_{ij} \Ex X^2_{ij} } \\
    \lesssim^\eta \max_{I\in \Ikc}  \sup_{t\in T}\sqrt{\Log(d_A)}\sqrt{\frac{\sum_{i\in I} \sum_{j} a^2_{ij}t^2_{ij}}{k} }\leq \sqrt{\Log d_A }\MTI.
\end{align*}
If $l\leq k$ then by an analogous argument
\begin{multline*}
   \max_{I\in \Ikc,\,J\in \Ilc} \max_{\eta_i,\eta'_j=\pm 1} \sup_{t\in T}\frac{1}{\sqrt{kl}}\lv \sum_{i\in I} \sum_{j\in J} a_{ij}t_{ij} X_{ij} \eta_i \eta'_j  \rv_{\Log(n^2 (64d_A)^{k+l})} \\
   \lesssim^\eta \sqrt{\Log \, d_A}\MTJ.  
\end{multline*}

The assertion follows by Lemma \ref{lem:rozbij}.

\end{proof}

\begin{theorem}\label{prop:oszapodgaus}
Let $A=(a_{ij})_{i,j\leq n}$ be a symmetric matrix. Consider arbitrary $T\subset \eR^{n^2}$. Assume that $(X_{ij})_{i,j\leq n}$ 
 are independent, symmetric, $\eta$-subgaussian random variables. Then, we have
\begin{align}
    \Ex \sup_{t\in T} &\lv (a_{ij} X_{ij} t_{ij})_{i,j \leq n} \rv_{op} \lesssim^\eta \Log^{3/2}(d_A) \Biggl(\sup_{v,w \in B_2}\Ex \sup_{t\in T}\sum_{ij}a_{ij} X_{ij} t_{ij}v_iw_j +\RTX \nonumber \\
  & +\MTI +\MTJ+\Ex  \max_{ij}\sup_{t\in T} |a_{ij}t_{ij} X_{ij}|\Biggl). \label{eq:loctw1}
\end{align}
In particular
\begin{multline}
    \Ex \sup_{t\in T} \lv (a_{ij} X_{ij} t_{ij})_{i,j \leq n} \rv_{op}  \lesssim^\eta \Log^{3/2}(d_A) \Biggl(\sup_{ v,w\in B_2} \lv \sup_{t\in T}\sum_{ij} a_{ij} t_{ij} X_{ij} v_i w_j \rv_{\Log \, n} \\
   +\sup_{t\in T} \lv \AT \rv_{op}+\Ex \max_{ij} \sup_{t\in T} |a_{ij}t_{ij} X_{ij}|\Biggl). \label{eq:loctw2}
\end{multline}
\end{theorem}
\begin{proof}
W.l.o.g. we may assume that $a_{ii}=0$ for $i\leq n$ (see Remark \ref{rem:da}). Clearly
\[\MTI+\MTJ \leq 2 \sup_{t\in T} \lv \AT \rv_{op}, \]
so it is enough to prove \eqref{eq:loctw1}. For $v,w\in B_2$ and integers $k,l$ we define
  \[I_k(s):=\{i\leq n: e^{-k-1}<|s_i|\leq e^{-k}\},\ J_l(t)=\{j\leq n: e^{-l-1}<|t_j|\leq e^{-l}\}.\]
  Observe that for any $v,w \in B_2$, $k,l \in \eN$ and $t\in T$
  \[\sum_{i\in I_k(s),\, j\in J_l(t)} a_{ij}t_{ij}X_{ij}v_iw_j \leq e^{-k-l} \max_{\eta_i,\eta_j'=\pm 1} \sum_{i\in I_k(s),\,j\in J_l(t)} a_{ij}t_{ij}X_{ij}\eta_i\eta'_j ,\]
  therefore
  \[\sup_{v,w\in B_2} \sup_{t\in T} \sum_{ij} a_{ij}t_{ij}X_{ij}v_iw_j  \leq \sup_{v,w \in B_2} \sup_{t\in T}\sum_{k,l}e^{-k-l}\max_{\eta_i,\eta_j'=\pm 1} \sum_{i\in I_k(v),j\in J_l(w)} a_{ij}t_{ij}X_{ij}\eta_i\eta'_j. \]
For fixed $v,w\in B_2$ and $t\in T$, we have
\begin{align*}
  &\sum_k \sum_{l \geq k+\Log \, d_A} e^{-k-l} \max_{\eta_i, \eta_j'} \sum_{i\in I_k(v), j\in J_l(w)} a_{ij}t_{ij}X_{ij}\eta_i \eta'_j \\
  &\leq \sum_k e^{-k} \sum_{i\in I_k(v)}   \sum_j |a_{ij}t_{ij}X_{ij}| \sum_{l \geq k+\Log \, d_A}e^{-l}\1_{j\in J_l(w)}\\
  &\lesssim \sum_k  \sum_{i\in I_k(v)}   \sum_j  |a_{ij}t_{ij}X_{ij}|e^{-2k-\Log \, d_A}\leq \max_{ij}\sup_{t\in T} |a_{ij}t_{ij}X_{ij}|\sum_k  \sum_{i\in I_k(v)}e^{-2k} \frac{\sum_j \1_{a_{ij}\neq 0}}{d_A}  \\
  &\leq \max_{ij}\sup_{t\in T}|a_{ij}t_{ij}X_{ij}|\sum_k  \sum_{i\in I_k(v)} e^2 v^2_i= \max_{ij}\sup_{t\in T}|a_{ij}t_{ij}X_{ij}| e^2 \lv v \rv^2_2,
\end{align*}
where in the last inequality, we used the fact that $\sum_j \1_{a_{ij}\neq 0} \leq d_A$ and the definition of the set $I_k(v)$. Since the matrix $A$ is symmetric, we also have that $\sum_i \1_{a_{ij}\neq 0} \leq d_A$. Thus,  in the same way, we can show that 
\[\sup_{t\in T}\sum_l \sum_{k \geq l+\Log \, d_A} e^{-k-l} \max_{\eta_i, \eta_j'} \sum_{i\in I_k(v), j\in J_l(w)} a_{ij}t_{ij}X_{ij}\eta_i \eta'_j\lesssim \max_{ij}\sup_{t\in T}|a_{ij}t_{ij}X_{ij}|. \]
Moreover, for any $t\in T$ and $v,w\in B_2$ (see \eqref{eq:defxx})
\begin{align*}
   \sum_{k,l: |k-l|<\Log \, d_A}e^{-k-l} \max_{\eta_i, \eta_j'} \sum_{i\in I_k(v), j\in J_l(w)} a_{ij}t_{ij}X_{ij}\eta_i \eta'_j \leq \XX   \sum_{k,l: |k-l|<\Log \, d_A}e^{-k-l} \sqrt{|I_k(v)||J_l(w)|}.
\end{align*}
For any fixed $r$
\begin{align*}
    \sum_k e^{-k-(k-r)} \sqrt{|I_k(v)||J_{k+r}(w)|}&\leq \left(\sum_k e^{-2k} |I_k(v)| \right)^{1/2}\left(\sum_k e^{-2(k+r)}|J_{k+r}(w)| \right)^{1/2}\\
    &\leq e^2 \lv v \rv_2 \lv w \rv_2.
\end{align*}
Hence
\begin{align*}
  \sup_{v,w\in B_2} \sup_{t\in T} \sum_{ij} a_{ij}X_{ij}t_{ij}v_iw_j &\lesssim \Log \, d_A  \XX +    \max_{ij}\sup_{t\in T}|a_{ij}t_{ij}X_{ij}|.
\end{align*}
The assertion follows from taking expectations on both sides and applying Lemmas \ref{lem:pierwszy} and \ref{lem:drugi}.
\end{proof}

\section{Proof of Theorem \ref{tw:glow}}
We prove Theorem \ref{tw:glow} by induction on $\wym=\lceil 2\cdot \log_2 \kappa \rceil$ (we recall that $\kappa$ is such that $\lv X_{ij} \rv_{2p} \leq \kappa \lv X_{ij} \rv_p$ for any $p \geq 1$). If $\wym=1$ then the random variables $(X_{ij})_{i,j\leq n}$ satisfies $\lv X_{ij} \rv_{2p} \leq \sqrt{2} \lv X_{ij} \rv_p$ for any $p\geq 1$. This easily implies that $\lv X_{ij} \rv_{p} \leq C \sqrt{p} \lv X_{ij} \rv_1$. Therefore, the base case is that of $C$-subgaussian random variables (where $C$ does not depend on $\kappa$). This is the content of the next theorem.
\begin{theorem}\label{tw:bazaindukcji}
    Assume that $X_{ij}$ are independent, symmetric, $C$-subgaussian random variables, such that $\Ex |X_{ij}|=1$. Then for any symmetric matrix $(a_{ij})_{i,j\leq n}$, we have
    \begin{align*}
    \Ex \lv \AX \rv_{op}\lesssim \Log^{3/2}(d_A)\left(\mia+ \RX\right).
\end{align*}
\end{theorem}
\begin{proof}
The assertion follows by invoking Proposition \ref{prop:oszapodgaus} with $T=\{(1)_{i,j \leq n}\}$ ($T$ contains only one matrix with ones everywhere). It is enough to observe that,  by Lemma \ref{lem:oszmaxpodex}
\[\Ex  \max_{ij} |a_{ij} X_{ij}| \lesssim \RX . \]
\end{proof}
The main idea of the proof of Theorem \ref{tw:glow} is that if $\wym \in \eN$, and the random variables $(X_{ij})_{i,j \leq n}$ satisfy \eqref{eq:*} with $\kappa=2^{\wym/2}$, then each $X_{ij}$ is a product of at most $\wym$ subgaussian random variables. We will condition on a certain group of them and then apply the induction assumption. Proposition \ref{prop:oszapodgaus} will handle the induction step. The key point is the analysis of the function
\begin{equation*}
    N^X_{ij}(t)=-\ln \Pro(|X_{ij}|\geq t).
\end{equation*}
It turns out that if $X_{ij}$ satisfies \eqref{eq:*}, then $N^X_{ij}$ has certain asymptotics.
\begin{lemma}\cite[Lemma 3.1]{ja}\label{lem:wlasnosciN}
   Let $X$ be any random variable from the $SRV(\kappa)$ class such that $\Ex |X|=1$. Let $N(t)=-\ln \Pro(|X|\geq t)$. Then there exists $C(\kappa)$ such that for any $x,t\geq 1$ we have $N(C(\kappa)tx)\geq t^{\frac{1}{\Log_2 \kappa}}N(x)$. In particular, $N(C(\kappa)t)\geq t^{\frac{1}{\Log_2 \kappa}}$ for $t\geq 1$.
\end{lemma}
\begin{proof}
 The first part of the Lemma was proven in \cite{ja} (Lemma 3.1 therein) for random variables that satisfy \eqref{eq:*} with $\kappa=2^\wym$ for some $\wym\in \eN$. However, this specific value of the parameter is irrelevant. For this reason, we skip the proof. The second part of the lemma follows from the first one and Markov's inequality, since
\[N(e)\geq -\ln \frac{\Ex |X|}{e}=1.\]
\end{proof}

The following Lemma is just a slight modification of the idea which was used in \cite{ja} (Lemma 3.3 therein).
\begin{lemma}\label{lem:konstrukcjazmiennych}
 Let $X$ be a symmetric random variable from the $SRV(2^{\wym/2})$ class such that $\Ex |X|=1$. Let $X^{(1)},\ldots,X^{(\wym)}$ be i.i.d. random variables, distributed as $\eps |X|^{1/\wym}$, where $\eps$  is the Rademacher random variable (a symmetric $\pm 1$ random variable) independent of $X$. Then, there exists a constant, $C=C(\wym)$ such that 
 \begin{enumerate}
     \item   $\Pro\left(\left|\prod_{k=1}^\wym X^{(k)}\right|\geq t\right) \geq  \Pro(|X|\geq C(\wym)t) $            
     \item    $ \Pro(|X|\geq t) \geq \Pro\left(\left|\prod_{k=1}^\wym X^{(k)}\right|\geq C(\wym) t \right)$,                      
     \item $X^{(1)},\ldots,X^{(\wym)}$ are $C(\wym)$-subgaussian,
     \item $\Ex |X^{(1)}|\rown ^\wym 1$.
 \end{enumerate}
\end{lemma}
\begin{proof}
Let $N(t)=-\ln \Pro(|X|\geq t)$. Observe that for $t\geq 1$
\begin{align*}
    \Pro\left(\left|\prod_{k=1}^\wym X^{(k)}\right|\geq t\right) &\geq    \prod_{k=1}^\wym \Pro\left(\left| X^{(k)}\right|\geq t^{1/\wym}\right)=\left(\Pro(|X|\geq t)\right)^\wym=e^{-\wym N(t)}\geq e^{-N(C\wym^{\wym/2}t)}\nonumber \\
    &=\Pro(|X|\geq C\wym^{\wym/2}t),
\end{align*}
where in the last inequality we used Lemma \ref{lem:wlasnosciN}. The same lemma implies that, for $t\geq 1$ and a large enough constant $C$  (which depends only on $\wym$)
\begin{align*}
    \Pro\left(\left|\prod_{k=1}^\wym X^{(k)}\right|\geq C^\wym t \right)&\leq \wym\Pro\left(|X^{(1)}|\geq Ct^{1/\wym}\right)= e^{\log \wym - N(C^\wym t)}\leq e^{ -\frac{N(C^\wym t)}{2}} \leq e^{-N(t)}\nonumber \\
    &=\Pro(|X|\geq t).
\end{align*}
 Since $X$ satisfies \eqref{eq:*} with $\kappa=2^{\wym/2}$ and $\Ex |X|=1$, it is easy to see that
\[\forall_{p\geq 1} \lv X \rv_p \leq Cp^{\wym/2}.\]
Thus, by Fact \ref{fa:oszmomwyk}
\[\Ex e^{c(\wym)(X^{(1)})^2}=\Ex e^{c(\wym)|X|^{2/\wym}} <\infty. \]
Hence $X^{(1)},\ldots,X^{(\wym)}$ are subgaussian with a constant $C(\wym)$.  Finally, by Jensen's inequality
\[\Ex |X^{(1)}|= \Ex |X|^{1/\wym} \leq \left( \Ex |X|\right)^{1/\wym}=1. \]
Using a similar argument, ($X^{(1)}$ is $C(\wym)$ subgaussian so it satisfies \eqref{eq:*} with $\kappa(\wym)$)
\[\Ex |X^{(1)}|\gtrsim^\wym  \lv X^{1} \rv_\wym= (\Ex|X|)^{1/\wym}=1.\]
\end{proof}
We now show that, in our setting, the random variables $X_{ij}$ can be replaced
by the products of subgaussian random variables.
\begin{corollary}\label{cor:zmienzmienne}
Let $(X_{ij})_{i,j \leq n}$ be independent random variables belonging to the 
SRV$(2^{\wym/2})$ class, normalized so that $\Ex |X_{ij}| = 1$ for each $i,j \leq n$.
Let $(X^{(k)}_{ij})_{k \leq d,\; i,j \leq n}$ be the independent random variables
associated with $(X_{ij})_{i,j \leq n}$ and constructed in Lemma~\ref{lem:konstrukcjazmiennych}.
Then, for any matrix $(a_{ij})_{i,j \leq n}$ and any $p \geq 1$,
    \begin{equation}\label{eq:pormom}
        \lv \sum_{ij} a_{ij} X_{ij} \rv_p \rown^\wym \lv \sum_{ij} a_{ij} \prod_{k=1}^\wym X^{(k)}_{ij}   \rv_p.
    \end{equation}
    Moreover, for any set $S\subset \eR^{n^2}$
    \begin{equation}\label{eq:loc12}
       \Ex \sup_{s\in S} \sum_{ij} s_{ij} X_{ij}  \rown^\wym \Ex \sup_{s\in S} \sum_{ij} s_{ij} \prod_{k=1}^\wym X^{(k)}_{ij}.   
    \end{equation}
\end{corollary}
\begin{proof}
Equations \eqref{eq:pormom} and \eqref{eq:loc12} are both simple consequences of Lemmas \ref{lem:porsup} and \ref{lem:konstrukcjazmiennych}.
\end{proof}

We will now prove Theorem \ref{tw:glow} in the special case when the random variables $(X_{ij})_{i,j \leq n}$ are the products of $\wym$ independent $C(\wym)$ subgaussian random variables.
\begin{lemma}\label{lem:twierdzeniedlaproduktow}
 Let $r \in \mathbb{N}$. Assume that $(X^{(k)}_{ij})_{k \leq \wym,\; i,j \leq n}$ 
are independent $C(\wym)$ subgaussian random variables such that 
$\Ex |X^{(k)}_{ij}| \sim^{\wym} 1$ for each $k \leq \wym$ and $i,j \leq n$.
Then, for any symmetric matrix $(a_{ij})_{i,j \leq n}$,
we have
    \begin{equation}\label{eq:zalinduk}
\Ex \sup_{v,w \in B_2} \sum_{ij} a_{ij} \left( \prod_{k=1}^\wym \Xk_{ij}\right) v_i w_j \lesssim^\wym \left(\Log(d_A)\right)^{C(\wym)}\left(\mia+R_{\Xld}(A)\right),
    \end{equation}
    where $X^{(\leq l)}_{ij}:=\prod_{k=1}^l \Xk_{ij}$ for $l=1,\ldots,\wym$, and we recall
    \[R_{X^{(\leq l)}}(A)=R_{X^{(\leq l)}}(A,\Log \, n)=\sup_{v,w\in B_2} \lv \sum_{ij} a_{ij} \left( \prod_{k=1}^l \Xk_{ij}\right) v_i w_j \rv_{\Log \, n}. \]
\end{lemma}
\begin{proof}
   Theorem \ref{tw:bazaindukcji} solves the case when $\wym=1$. Assume that the assertion holds for $\wym\geq 2$, i.e., \eqref{eq:zalinduk} is true. Let  $T\subset \eR^{n^2}$ be the set  for which \eqref{eq:wlasnoscT} holds for
\[(\Xld_{ij})_{i,j \leq n}=(\prod_{k=1}^\wym \Xk_{ij})_{i,j\leq n}\] 
in place of  $X_{ij}$. Such set exists, since random variables $(\Xld_{ij})_{i,j \leq n}$ belong to the SRV$(C 2^{\wym/2})$ class (this is an easy application of Theorem \ref{thm:podgausschar}). In our notation, this means that, for any matrix $M=(m_{ij})_{i,j \leq n}$, we have
\[R_{\Xld}(M) \rown^\wym \sup_{v,w\in B_2} \sup_{t\in T}\sum_{ij} m_{ij} t_{ij} v_i w_j. \]
In particular, for the random matrix  $M=A \odot \Xd:=(a_{ij}\Xd_{ij})_{i,j \leq n}$, we have
\begin{align}
    \Ex  R_{\Xld}(A \odot \Xd)  \rown^\wym  \Ex\sup_{v,w\in B_2} \sup_{t\in T}\sum_{ij} a_{ij} t_{ij} \Xd_{ij} v_i w_j. \label{eq:locpikachu*}
\end{align}
If for some $i,j \leq n$ we have $a_{ij}=0$, then $(A \odot \Xd)_{ij}=0$ (the entry in the $i$-th row and $j$-th column of the matrix $A \odot \Xd$). As a result
\begin{equation}\label{eq:oszda}
d_{A \odot \Xd} \leq d_A.
\end{equation}
   We use the induction assumption i.e. \eqref{eq:zalinduk}, conditioned on the random matrix $=A \odot \Xd$ to upper bound
    \begin{align*}
      &\Ex \sup_{v,w \in B_2} \sum_{ij} a_{ij} \left( \prod_{k=1}^{\wym+1} \Xk_{ij}\right) v_i w_j = \Ex \sup_{v,w \in B_2} \sum_{ij} (A \odot \Xd)_{ij} \left( \prod_{k=1}^{\wym} \Xk_{ij}\right) v_i w_j  \\
      &\lesssim^\wym \Ex \Log^{C(\wym)}  \left(d_{A \odot \Xd} \right)  \left( \max_i \sqrt{\sum_j (A \odot \Xd)^2_{ij}}+ R_{\Xld}(A \odot \Xd) \right)\\
      &\lesssim^\wym \left(\Log(d_A)\right)^{C(\wym)} \left(\Ex \max_i \sqrt{\sum_j a^2_{ij} (\Xd_{ij})^2}+\Ex\sup_{v,w\in B_2} \sup_{t\in T}\sum_{ij} a_{ij} t_{ij} \Xd_{ij} v_i w_j \right),
    \end{align*}
    where the last line follows from \eqref{eq:locpikachu*} and \eqref{eq:oszda}. Since $\max_i \sqrt{\sum_j b^2_{ij}}\leq \lv (b_{ij})\rv_{op}$, Theorem \ref{tw:bazaindukcji} implies that
\begin{align*}
   \Ex \max_i \sqrt{\sum_j a^2_{ij} (\Xd_{ij})^2} \lesssim \Log^{3/2} (d_A)\left( R_{X^{\wym+1}}(A) +\mia\right).
\end{align*}
 The above can be shown directly by a simple argument (even without the $\Log(d_A)$ part). We refer interested readers to \cite[Proposition 4.4]{latswiat}).
By Lemmas \ref{lem:turbojensen} and then \ref{lem:zmnienwspol} (we recall that $\Ex |\Xk_{ij}| \geq c(\wym)>0$)
\begin{align*}
 R_{X^{(\leq \wym+1)}}(A)&\geq \sup_{v,w\in B_2} \lv \sum_{ij} a_{ij}  \Xd_{ij} v_i w_j \left(\Ex  \prod_{k=1}^{\wym}|\Xk_{ij}|\right)\rv_{\Log \, n} \\
  &\gtrsim^\wym \sup_{v,w\in B_2} \lv \sum_{ij} a_{ij}  \Xd_{ij} v_i w_j \rv_{\Log \, n}= R_{X^{(\wym+1)}}(A).
\end{align*}
Hence, it is enough to show that
\begin{align}
 \Ex\sup_{v,w\in B_2} \sup_{t\in T}\sum_{ij} a_{ij} t_{ij} \Xd_{ij} v_i w_j  &\lesssim^\wym  \Log^{3/2}(d_A)R_{X^{\leq \wym+1}}(A).\label{eq:cel1} 
\end{align}
Theorem \ref{prop:oszapodgaus} (more precisely \eqref{eq:loctw2}) yields
\begin{multline}
    \Ex\sup_{v,w\in B_2} \sup_{t\in T}\sum_{ij} a_{ij} t_{ij} \Xd_{ij} v_i w_j \lesssim^\wym \Log^{(3/2)}(d_A) \Bigg(\sup_{v,w \in B_2} \lv \sup_{t\in T} \sum_{ij} a_{ij} t_{ij} \Xd_{ij} \rv_{\Log \, n}  \\
 +\sup_{t\in T} \lv \AT\rv_{op}+\Ex \sup_{t\in T} \max_{ij} |a_{ij}t_{ij} \Xd_{ij}|\Bigg).\label{eq:loc1a}
\end{multline}
By   \eqref{eq:suptpodmomentem}
\begin{align}
 \sup_{v,w\in B_2} \lv \sup_{t\in T} \sum_{ij} a_{ij} t_{ij} \Xd_{ij} v_i w_j \rv_{\Log \, n}\rown^\wym   \sup_{v,w\in B_2} \lv \sum_{ij} a_{ij} \prod_{k=1}^{\wym+1} \Xk_{ij} v_i w_j \rv_{\Log \, n}=R_{X^{(\leq \wym+1)}}(A).
 \label{eq:loc1b}
\end{align}

Again by \eqref{eq:wlasnoscT} and Lemmas \ref{lem:turbojensen} and then \ref{lem:zmnienwspol} (recalling that $\Ex |\Xd_{ij}| \geq c(\wym)>0$)
\begin{align}
\sup_{t\in T} \lv \AT \rv_{op}&= \sup_{v,w\in B_2}\sup_{t\in T} \sum_{ij} a_{ij}t_{ij}  v_iw_j \rown^\wym   \sup_{v,w\in B_2} \lv \sum_{ij} a_{ij} \prod_{k=1}^{\wym} \Xk_{ij} v_i w_j \rv_{\Log \, n} \nonumber \\
 &\lesssim^\wym \sup_{v,w\in B_2} \lv \sum_{ij} a_{ij} \prod_{k=1}^{\wym+1} \Xk_{ij} v_i w_j \rv_{\Log \, n}=R_{X^{(\leq \wym+1)}}(A).\label{eq:loc1d}
\end{align}

 Lastly, by Fact \ref{lem:oszmaxpodex} and \eqref{eq:wlasnoscT}
\begin{align}
    \Ex  \max_{ij}\sup_{t\in T} |a_{ij}t_{ij} \Xd_{ij}|&\rown^\wym \Ex \max_{ij} |a_{ij}||\Xd_{ij}| \lv \prod_{k=1}^\wym \Xk_{ij}\rv_{\Log \, n} \nonumber \\
    &\lesssim \max_{ij} \lv a_{ij} \Xd_{ij}\lv \prod_{k=1}^\wym \Xk_{ij}\rv_{\Log \, n}  \rv_{\Log \, n}=\max_{ij} \lv a_{ij}\prod_{k=1}^{\wym+1}\Xk_{ij}\rv_{\Log \, n} 
    \nonumber \\  
    &\leq \sup_{v,w\in B_2 }\lv \sum_{ij} a_{ij}v_i w_j\prod_{k=1}^{\wym+1}\Xk_{ij}\rv_{\Log \, n}=R_{X^{\leq \wym+1}}(A) . \label{eq:loc1e}
\end{align}
Clearly, \eqref{eq:cel1} is a consequence of \eqref{eq:loc1a}-\eqref{eq:loc1e}.
 
\end{proof}

\begin{proof}[proof of Theorem \ref{tw:glow}]  
We may assume that the random variables $(X_{ij})_{ij}$ satisfy \eqref{eq:*} with $\kappa=2^{\wym/2}$ for certain $\wym\in \eN$.  Let $(\Xk_{ij})_{k\leq \wym;i,j\leq n}$ be the random variables constructed in Lemma \ref{lem:konstrukcjazmiennych}.  Using Corollary \ref{cor:zmienzmienne} and then Lemma \ref{lem:twierdzeniedlaproduktow}, we obtain
    \begin{multline*}
\Ex \lv \AX\rv_{op} =\Ex \sup_{v,w \in B_2} \sum_{ij} a_{ij}X_{ij}v_iw_j \rown^\wym \Ex \sup_{v,w \in B_2} \sum_{ij} a_{ij}\left( \prod_{k=1}^\wym \Xk _{ij}\right)v_iw_j \\
\lesssim^\wym  
    \left(\Log(d_A)\right)^{C(\wym)}\left(\mia+R_{\Xld}(A)\right).
    \end{multline*}

   By Corollary \ref{cor:zmienzmienne}
    \[R_{\Xld}(A)\rown^\wym \RX,\]
    so \eqref{eq:twglownezda} follows.
\end{proof}






\begin{biblist}
\begin{bibdiv}
\bibliographystyle{amsplain}




\bib{adlat}{article}{
   author={Adamczak, Rados{\l}aw},
   author={Lata{\l}a, Rafa{\l}},
   title={Tail and moment estimates for chaoses generated by symmetric
   random variables with logarithmically concave tails},
   journal={Ann. Inst. Henri Poincar\'e Probab. Stat.},
   volume={48},
   date={2012},
   number={4},
   pages={1103--1136},
}











\bib{mac:los}{book}{
   author={Anderson, Greg W.},
   author={Guionnet, Alice},
   author={Zeitouni, Ofer},
   title={An introduction to random matrices},
   series={Cambridge Studies in Advanced Mathematics},
   volume={118},
   publisher={Cambridge University Press, Cambridge},
   date={2010},
   pages={xiv+492},
   isbn={978-0-521-19452-5},
}
























\bib{latop}{article}{
   author={Lata{\l}a, Rafa{\l}},
   title={On the spectral norm of Rademacher matrices},
   journal={arXiv:2405.13656},
   date={2024},
}






\bib{strzelec}{article}{
   author={Lata\l a, Rafa\l },
   author={Strzelecka, Marta},
   title={Comparison of weak and strong moments for vectors with independent
   coordinates},
   journal={Mathematika},
   volume={64},
   date={2018},
   number={1},
   pages={211--229},
}


\bib{latswiat}{article}{
   author={Lata\l a, Rafa\l},
   author={\'Swi\polhk atkowski, Witold},
   title={Norms of randomized circulant matrices},
   journal={Electron. J. Probab.},
   volume={27},
   date={2022},
   pages={Paper No. 80, 23},
}






\bib{lathan}{article}{
   author={Lata\l a, Rafa\l},
   author={van Handel, Ramon},
   author={Youssef, Pierre},
   title={The dimension-free structure of nonhomogeneous random matrices},
   journal={Invent. Math.},
   volume={214},
   date={2018},
   number={3},
   pages={1031--1080},
}



\bib{proinban}{book}{
   author={Ledoux, Michel},
   author={Talagrand, Michel},
   title={Probability in Banach spaces},
   series={Ergebnisse der Mathematik und ihrer Grenzgebiete (3) [Results in
   Mathematics and Related Areas (3)]},
   volume={23},
   publisher={Springer-Verlag, Berlin},
   date={1991},
}



\bib{d2}{article}{
   author={Meller, Rafa\l },
   title={Tail and moment estimates for a class of random chaoses of order
   two},
   journal={Studia Math.},
   volume={249},
   date={2019},
   number={1},
   pages={1--32},
}

\bib{ja}{article}{
   author={Meller, Rafa\l },
   title={Two-sided moment estimates for a class of nonnegative chaoses},
   journal={Statist. Probab. Lett.},
   volume={119},
   date={2016},
   pages={213--219},
}






\bib{seg}{article}{
   author={Seginer, Yoav},
   title={The expected norm of random matrices},
   journal={Combin. Probab. Comput.},
   volume={9},
   date={2000},
   number={2},
   pages={149--166},
}



\bib{tal}{book}{
   author={Talagrand, Michel},
   title={Upper and lower bounds for stochastic processes---decomposition
   theorems},
   series={Ergebnisse der Mathematik und ihrer Grenzgebiete. 3. Folge. A
   Series of Modern Surveys in Mathematics [Results in Mathematics and
   Related Areas. 3rd Series. A Series of Modern Surveys in Mathematics]},
   volume={60},
   note={Second edition [of  3184689]},
   publisher={Springer, Cham},
   date={[2021] \copyright 2021},
}
	



\end{bibdiv}

\end{biblist}
\end{document}